\documentclass[submission%
]{dmtcs}

\usepackage{amssymb,amsmath}
\usepackage{shuffle}
\usepackage{gensymb}
\usepackage{latexsym}
\usepackage{tikz}
\usepackage[utf8]{inputenc}
\usepackage{hyperref}
\usepackage{bm}
\usepackage[all]{xy}
\usepackage{todonotes}
\usepackage{subfigure}
\usepackage{tikz}

\newtheorem{example}{Example}[section]

\newtheorem{remark}[example]{Remark}
\newtheorem{theorem}[example]{Theorem}

\newtheorem{corollary}[example]{Corollary}

\newtheorem{definition}[example]{Definition}
\newtheorem{proposition}[example]{Proposition}

\newtheorem{lemma}[example]{Lemma}

\renewcommand{\leq}{\leqslant}

\def\bd{\gamma}
\def\ex{\text{ex}}

\def\qed{\hspace{3.5mm} \hfill \vbox{\hrule height 3pt depth 2 pt width 2mm}
\bigskip}
\def\X{{\mathbb X}}

\def\Sym{{\bf Sym}}
\def\QSym{{\it QSym}}

\def\WQSym{{\bf WQSym}}

\def\<{\langle}
\def\>{\rangle}

\def\CC{{\mathcal C}}

\def\Q{{\mathbb Q}}
\def\Z{{\mathbb Z}}
\def\N{{\mathbb N}}

\def\A{{\mathbb A}}

\def\y{{\bf y}}
\def\x{{\bf x}}
\def\bc{{\bf c}}
\def\bv{{\bf v}}
\def\bw{{\bf w}}
\def\pp{{\bm p}}
\def\qq{{\bm q}}

\def\Sx{ \mathcal{S}_x} 
\def\Spq{ \mathcal{S}_{pq}} 

\newdimen\squaresize
\newdimen\thickness         

\def\ie{{\em i.e. }}

\def\gf#1#2{\genfrac{}{}{0pt}{}{#1}{#2}}

\author[J.-C. Aval, V. Féray, J.-C. Novelli and J.-Y. Thibon]%
{Jean-Christophe Aval\addressmark{1},
  Valentin Féray\addressmark{2},
  Jean-Christophe Novelli\addressmark{3},
  \and Jean-Yves Thibon\addressmark{3}.}

\title{Super quasi-symmetric functions via Young diagrams}

\address{\addressmark{1}LaBRI, CNRS, Université de Bordeaux, 351 cours de la Libération, Talence, France\\
  \addressmark{2}Institüt für Mathematik, Universität Zürich, Winterthurerstrasse 190, Zürich, Switzerland\\
  \addressmark{3}LIGM, Université Paris-Est Marne-La-Vallée, 5 boulevard Descartes, Champs-sur-Marne, France}

\begin{document}

\maketitle

\begin{abstract}
\paragraph{Abstract.}
We consider the multivariate generating series $F_P$ of $P$-partitions
in infinitely many variables $x_1, x_2, \dots$.
For some family of ranked posets $P$, it is natural to consider
an analog $N_P$ with two infinite alphabets.
When we collapse these two alphabets, we trivially recover $F_P$.
Our main result is the converse, that is, the explicit construction
of a map sending back $F_P$ onto $N_P$.
We also give a noncommutative analog of the latter.
An application is the construction of a basis
of $\WQSym$ with a non-negative multiplication table,
which lifts a basis of $\QSym$ introduced by K.~Luoto.

\paragraph{R\'esum\'e.}
Nous consid\'erons la s\'erie g\'en\'eratrice multi-vari\'ee $F_P$ des
$P$-partitions en un ensemble infini de variables $x_1, x_2, \dots$.
Pour une certaine famille d'ensembles ordonnés $P$,
on peut consid\'erer un analogue $N_P$ en deux
ensembles de variables.
En \'egalant les deux alphabets, on retrouve \'evidemment $F_P$.
Notre r\'esultat principal est la r\'eciproque de cela~: nous montrons qu'il
existe une op\'eration retournant $N_P$ \`a partir de $F_P$.
Nous donnons aussi un analogue non-commutatif de cette op\'eration.
Nous obtenons ainsi une nouvelle base de $\WQSym$, base qui rel\`eve une base
de K.~Luoto et dont les coefficients de structure sont positifs.
\end{abstract}

This article is an extended abstract of \cite{LongVersion} in the sense
that most results and proofs given here are also in \cite{LongVersion}.
Nevertheless, the focus here is very different from the focus in
\cite{LongVersion}.

\section{Introduction}
Consider a ranked poset $P$ on $n$ elements.
We consider non-decreasing functions $r$ from $P$
to the set $\N$ of positive integers, with the additional condition that $r(x) < r(y)$
whenever $x <_P y$ and $x$ has odd height in $P$
(we say that such functions satisfy the {\em order condition}).
These correspond to $P$-partition for some labeling of the elements of $P$
and thus fit in the general context studied by R.~Stanley in \cite{StOrderedStructure}.

Following I.~Gessel \cite{Ges}, we consider the multivariate
 generating series
 \begin{equation}\label{eq:def_FP}
  F_P(x_1, x_2, \cdots ) = \sum_{r} \prod_{i \in P} x_{r(i)},
  \end{equation}
where the sum runs over functions $r$ from $P$ to $\N$
satisfying the order condition.
This series in infinitely 
many variables turns out to be a quasi-symmetric function
(in fact, quasi-symmetric functions were origi-\pagebreak\\
nally introduced by I.~Gessel
in \cite{Ges} to give an algebraic framework 
to these multivariate generating series of $P$-partitions).

Notice that, in our framework, elements of even and odd heights in the poset $P$
play different roles.
Therefore it is also natural
to consider generating series in two alphabets $p_1, p_2, \cdots$ and $q_1,q_2, \dots$\footnote{In the case of posets of height $1$,
this generating series in two variables appears in representation theory of symmetric groups,
see \cite[Definition 2.2.1]{Fer09}.} :
 \begin{equation}\label{eq:def_NP}
 N_P\left( \begin{array}{ccc}
    p_1 & p_2 & \dots\\
    q_1 & q_2 & \dots
\end{array} \right) = \sum_{r} \left( \prod_{v_0 \in V_0} p_{r(v_0)}
\prod_{v_1 \in V_1} q_{r(v_1)} \right),
  \end{equation}
where $V_0$, resp. $V_1$, denotes the elements of $P$ of even, resp. odd, heights
(we shall use this notation throughout the paper)
and where the sum runs
over functions $r$ from $P$ to $\N$ satisfying the order condition.
Then $N_P$ is a quasi-symmetric in two sets of variables:
by analogy with the work of Stembridge \cite{StembridgeSuperSym},
we call these functions {\em super quasi-symmetric functions}.
An example of $F_P$ and $N_P$ is given in Section \ref{sec:P-part}.
Clearly, setting $p_i=q_i=x_i$ in $N_P$ allows us to recover $F_P$.

Our main result is a converse of this simple remark.
More precisely, we construct explicitly a map from quasi-symmetric functions
to series in two infinite alphabets which sends $F_P$ onto $N_P$.
Our construction is naturally stated in the language of Hopf algebras
calculus, as explained in Section \ref{sec:QSym_and_Virtual_Alphabets}.
Fix some integer $m$.
Define the virtual alphabet
\begin{equation}\label{eq:def_Virtual_X}
\X_m = \ominus (x_1) \oplus (x_2) \ominus (x_3) \cdots \ominus (x_{2m+1}), 
\end{equation}
where the $x_i$ are the following linear combinations of $p_i$ and $q_i$
\begin{equation}\label{eq:x_pq}
\left\{ \begin{array}{l}
x_{2i+1} = q_{i+1} + \cdots + q_m + p_{i+1} + \cdots + p_{m+1} ;  \\
x_{2i} = q_{i} + \cdots + q_m + p_{i+1} + \cdots + p_{m+1}.
\end{array} \right.
\end{equation}
Our main theorem is the following :
\begin{theorem}\label{thm:main}
Let $P$ be a ranked poset and $m$ a non-negative integer.
With the notations above,
\begin{equation}\label{eq:main}
N_P \left( \begin{array}{cccccc}
    p_1 & \dots & p_m & p_{m+1} & 0  & \dots \\
    q_1 & \dots & q_m & 0 & 0 & \dots
\end{array} \right) 
= (-1)^{|V_0|} F_P(\X_m).
\end{equation}
\end{theorem}
Note that the left-hand side for all values of $m$
determines the series $N_P$,
so, as claimed, the theorem allows to reconstruct $N_P$ explicitly from $F_P$.
We prove it at the end of Section \ref{sec:P-part}.

Let us say a word about the proof.
Notably, it does not involve any computation,
but relies on the structure of the spaces of solutions of 
two functional equations
presented in Section \ref{sec:Functional_Equations}.
These functional equations come from the analysis of smooth functions
on Young diagrams (and formula \eqref{eq:x_pq} has a transparent interpretation
in terms of Young diagrams).

Another remarkable feature of our result is that it readily extends to a
noncommutative framework.
In Section \ref{sec:NC}, we state a noncommutative analog
of Theorem \ref{thm:main}.

We end the article by an application of our noncommutative result.
In \cite{Luoto}, K.~Luoto introduced a basis of $\QSym$
with interesting applications to matroid theory.
Here, we consider the natural noncommutative lift of his basis in
$\WQSym$, the natural noncommutative analog of $\QSym$.
Thanks to our result,
we are able to show that this family is linearly independent
and hence, a basis of $\WQSym$.
This new basis has nice properties, as the fact that its multiplication table
contains only nonnegative integers.

The linear independence of this new basis is easily proved using 
our theorem involving Hopf algebra calculus,
while we have not been able to find an elementary proof of it.

\section{Definitions and notations}

\subsection{Stable polynomials}

A {\em stable} homogeneous
polynomial of degree $d$ is a sequence $R=(R_n(x_1,\ldots,x_n))_{n\ge 0}$ of
homogeneous polynomials of degree $d$ such that
$R_{n+1}(x_1,\ldots,x_n,0)=R_n(x_1,\ldots,x_n)$.
Intuitively, it is nothing else than a polynomial in infinitely many variables
$x_1,x_2,\cdots$.
Their set will be denoted by $\Q[X]$, where $X$ is the infinite variable set
$X=\{x_1,x_2,\dots\}$ (which should not be confused with the virtual alphabet
$\X_m$ defined by \eqref{eq:def_Virtual_X}).

In our framework, it is more natural to define
 stable polynomials by a sequence of polynomials in
an odd number of variables $(R_{2m+1})_{m \ge 0}$ such that
$R_{2m+1}(x_1,\dots,x_{2m-1},0,0)=R_{2m-1}(x_1,\dots,x_{2m-1}).$
This is not an issue, as such a sequence can be extended in a unique way to a
stable sequence $(R_n)_{n \ge 0}$ by setting
$R_{2m}(x_1,\dots,x_{2m}) = R_{2m+1}(x_1,\dots,x_{2m},0).$
\medskip

In the same spirit, we define an element of $\Q[\bm{p},\bm{q}]$ as a sequence
$(h_m)_{m \ge 0}$, where each $h_m$ is a polynomial in the $2m+1$ variables
$p_1,\dots,p_m,p_{m+1},q_1,\dots,q_m$
satisfying the stability property
\begin{equation}
h_{m+1}
\left( \begin{array}{ccccc}
    p_1 & \dots & p_m & p_{m+1} & 0\\
    q_1 & \dots & q_m & 0
\end{array} \right) = 
h_m
\left( \begin{array}{cccc}
    p_1 & \dots & p_m &p_{m+1}\\
    q_1 & \dots & q_m
\end{array} \right).
\end{equation}
Clearly, $N_P$ can be seen as an element of $\Q[\bm{p},\bm{q}]$.

\subsection{Quasi-symmetric functions and Hopf algebra calculus}
\label{sec:QSym_and_Virtual_Alphabets}

Quasi-symmetric functions were introduced by I.~Gessel in
relation with multivariate generating series of $P$-partition \cite{Ges}
and may be seen as a generalization of the notion of symmetric functions.
A comprehensive survey can be found in \cite{QSymBook}.

A {\em composition} of $n$ is a sequence $I=(i_1,i_2,\dots,i_r)$
of positive integers, whose sum is equal to $n$.
We denote by $\CC$ the set of all compositions 
(that is all compositions of all integers $n$).


The algebra $\QSym$ of quasi-symmetric functions is a subalgebra of the algebra
$\Q[X]$ of polynomials in the totally ordered commutative alphabet
$X=\{x_1,x_2,\dots\}.$
A basis of $\QSym$ is given by the {\em monomial
quasi-symmetric functions} $M$ indexed by compositions $I= (i_1,\dots, i_r)$,
where
\begin{equation}\label{eq:QSymM}
M_I = \sum_{a_1<\cdots<a_r} x_{a_1}^{i_1} \cdots x_{a_r}^{i_r} .
\end{equation}
By convention, $M_{()}=1$, where $()$ is the empty composition.
Note that the dimension
of $\QSym$ in degree $n$ is the number of compositions of $n$,
that is $2^{n-1}$.\medskip

An important property for Hopf algebras calculus is the duality
with the algebra $\Sym$ of {\em noncommutative symmetric functions}.
For a totally ordered alphabet $A=\{a_i|i\ge 1\}$ of noncommuting
variables, and $t$ an indeterminate, one sets
\begin{equation}
\sigma_t(A)=\prod_{i\ge 1}^\rightarrow (1-ta_i)^{-1} =\sum_{n\ge 0} S_n(A)t^n\,.
\end{equation}
The 
functions $S_n(A)$ generate a free associative algebra,
which is by definition $\Sym(A)$. One denotes by
$
S^I(A):= S_{i_1}S_{i_2}\cdots S_{i_r}
$
its natural basis. Actually, $\Sym(A)$ is the graded
dual of $QSym$. This can be deduced from the noncommutative Cauchy
formula
\begin{equation}
\prod_{i\ge 1}^\rightarrow \sigma_{x_i}(A) =
\prod_{i\ge 1}^\rightarrow \prod_{j\ge 1}^\rightarrow {(1-x_ia_j)}^{-1} =
\sum_{I \in \CC} M_I(X)S^I(A)
\end{equation}
which allows to identify $M_I$ with the dual basis of $S^I$ \cite{NCSF1,MR}.\medskip

Now, the evaluation of $M_I$ on the virtual alphabet $\X_m$ from the introduction
is implicitly defined by : 
\begin{equation}\label{eq:def_MIXm}
\sum_{I \in \CC} M_I(\X_m) S^I(A) = \prod_{1 \le i \le m+1}^\rightarrow \sigma_{x_i}(A)^{(-1)^i}\,.
\end{equation}
Then, we extend this definition to $F(\X_m)$, for any $F$ in $\QSym$ by linearity.

\begin{example}
Here are the functions $M_I(\X)$ for compositions $I$ of length at most 3:
\begin{align}
    M_{(k)}(\X)    &= - x_1^k + x_2^k - x_3^k +  \dots +x_{2m}^k - x_{2m+1}^k ; \label{eq:MkXm} \\
M_{(k,\ell)}(\X)   &= \sum_{i=1}^{2m+1} x_{2i+1}^{k+\ell}
                     + \sum_{1 \le i<j \le 2m+1} (-1)^{i+j} x_i^k x_j^\ell ;  \label{eq:MklXm}\\
M_{(k,\ell,m)}(\X) &= - \sum_{i=1}^{2m+1} x_{2i+1}^{k+\ell+m} + 
\sum_{\gf{i,j}{i<2j+1}} (-1)^i x_i^{k} x_{2j+1}^{\ell+m} \\
\nonumber & \quad +
          \sum_{\gf{i,j}{1\le 2i+1<j \le 2m+1}} (-1)^j x_{2i+1}^{k+\ell} x_j^m
+ \sum_{1 \le h<i<j \le 2m+1} (-1)^{h+i+j} x_h^k x_i^\ell x_j^m
\end{align}
\end{example}

\begin{remark}
{\rm
The definition of $M_I(\X_m)$ fits in the more general theory of Hopf algebras calculus,
which allows to define $F(\X)$, where $\X$ is written as a finite sum/difference of ordered alphabets.
In general $F \mapsto F(\X)$ is an algebra morphism, but we do not need this property here.
}
\end{remark}

\section{Two equivalent equations}
\label{sec:Functional_Equations}

\subsection{The functional equations}
The first functional equation that we consider is the following:
\begin{definition}
Let $\Sx$ be the space of stable polynomials $f=(f_{2m+1})_{m \ge 0}$ in
$\Q[X]$
such that, for each
$m \geq 1$ and each $1\le i \le 2m$, one has:
\begin{equation}
\label{eqfoncQS}
f_{2m+1}(x_1,\dots,x_{2m+1})\big|_{x_{i+1}=x_i}
  =f_{2m-1}(x_1,\dots,x_{i-1},x_{i+2},\ldots,x_{2m+1}).
\end{equation}
\end{definition}
Note that the left-hand side means that we substitute $x_{i+1}$ by $x_i$.
Then the equality must be understood as an equality between polynomials in
$x_1,\dots,x_{i-1},x_{i+2},\dots,x_n$.
In particular, the left-hand side must be independent of $x_i$.\medskip

The second functional equation we are interested in is the following:
\begin{definition}
Let $\Spq$ be the subspace of $\Q[\bm{p},\bm{q}]$ of elements $h=(h_m)_{m \ge 0}$ such
that, for all $m \ge 1$ and positive integer $i \le m$,
\begin{align}
    \left. h_m\left( \begin{array}{cccc} 
        p_1 & \dots & p_m &p_{m+1} \\
        q_1 & \dots & q_m
    \end{array} \right)\right|_{q_i=0}
    &=
    h_{m-1}\left( \begin{array}{cccccccc}
        p_1&\dots&p_{i-1}&p_i+p_{i+1}&\dots&p_m &p_{m+1}\\
        q_1&\dots&q_{i-1}&q_{i+1}&\dots&q_m
    \end{array} \right) ;
    \label{EqQZero} \\
    \left. h_m\left( \begin{array}{cccc} 
        p_1 & \dots & p_m &p_{m+1}\\
        q_1 & \dots & q_m
    \end{array} \right)\right|_{p_i=0}
    &=
    h_{m-1}\left( \begin{array}{cccccccc}
        p_1&\dots&p_{i-1}&p_{i+1}&\dots&p_m&p_{m+1}\\
        q_1&\dots&q_{i-1}+q_i&q_{i+1}&\dots&q_m
    \end{array} \right). \label{EqPZero}
\end{align}
By convention, in the second equation for $i=1$,
one should forget the column containing the undefined variables $p_0$ and $q_0$.
\end{definition}

\subsection{Origin: interlacing and multi-rectangular coordinates of Young
diagrams}

This section explains where our two equations come from.
It is written in an informal way
and can be safely skipped by a reader only interested in the proof of our theorem.\medskip

Consider a Young diagram $\lambda$ drawn with the Russian convention,
(\emph{i.e.}, draw it with the French convention, rotate it counterclockwise
by $45\degree$ and scale it by a factor $\sqrt{2}$).
Its border can be interpreted as the graph of a piecewise affine function $\omega_\lambda$.
We denote by $x_1, x_2,\dots,x_{2m+1}$ the abscissas of its local minima and
maxima in decreasing order, see Figure~\ref{FigRussian} (central part).

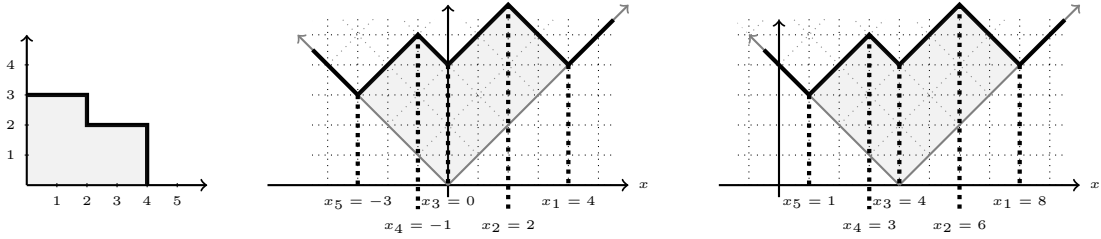
\begin{figure}[t]
    \[\begin{tikzpicture}[scale=.8]

      \begin{scope}[draw=black,scale=.5]

          \draw[->,thick] (0,0) -- (6,0);
          \foreach \x in {1, 2, 3, 4, 5}
              { \draw (\x, -2pt) node[anchor=north] {{\tiny{$\x$}}} -- (\x,
2pt); }

          \draw[->,thick] (0,0) -- (0,5);
          \foreach \y in {1, 2, 3, 4}
              { \draw (-2pt,\y) node[anchor=east] {{\tiny{$\y$}}} -- (2pt,\y); }

          \draw[ultra thick,draw=black] (4,0) -- (4,2) -- (2,2) -- (2,3) -- (0,3) ;
          \fill[fill=gray,opacity=0.1] (4,0) -- (4,2) -- (2,2) -- (2,3) -- (0,3) -- (0,0) -- cycle ;
      \end{scope}
 
\begin{scope}[xshift=7cm, yshift=-0cm, scale=0.5]
       \begin{scope}
          \clip (-4.5,0) rectangle (5.5,5.5);
          \draw[thin, dotted] (-6,0) grid (6,6);
          \begin{scope}[rotate=45,draw=gray,scale=sqrt(2)]
              \clip (0,0) rectangle (4.5,5.5);
              \draw[thin, dotted] (0,0) grid (6,6);
          \end{scope}
      \end{scope}

      \draw[->,thick] (-6,0) -- (6,0) node[anchor=west] {\tiny{$x$}};

      \draw[->,thick] (0,-0.4) -- (0,6);

%
%

\begin{scope}[draw=gray,rotate=45,scale=sqrt(2)]

          \draw[->,thick] (0,0) -- (6,0);

          \draw[->,thick] (0,0) -- (0,5);

          \draw[ultra thick,draw=black] (5.5,0) -- (4,0)  -- (4,2) -- (2,2) -- (2,3) --
          (0,3) -- (0,4.5) ;
          \fill[fill=gray,opacity=0.1] (4,0) -- (4,2) -- (2,2) -- (2,3) -- (0,3) -- (0,0) -- cycle ;

      \end{scope}

 \draw[ultra thick,dotted] (4,4) -- (4,0) node[anchor=north] {\tiny{$x_1=4$}};
 \draw[ultra thick,dotted] (2,6) -- (2,-.8) node[anchor=north] {\tiny{$x_2=2$}};
 \draw[ultra thick,dotted] (0,4) -- (0,0) node[anchor=north] {\tiny{$x_3=0$}};
 \draw[ultra thick,dotted] (-1,5) -- (-1,-.8) node[anchor=north] {\tiny{$x_4=-1$}};
 \draw[ultra thick,dotted] (-3,3) -- (-3,0) node[anchor=north] {\tiny{$x_5=-3$}};

\end{scope}

\begin{scope}[xshift=14.5cm, yshift=-0cm, scale=0.5]
       \begin{scope}
          \clip (-5.5,0) rectangle (5.5,5.5);
          \draw[thin, dotted] (-6,0) grid (6,6);
          \begin{scope}[rotate=45,draw=gray,scale=sqrt(2)]
              \clip (0,0) rectangle (4.5,5.5);
              \draw[thin, dotted] (0,0) grid (6,6);
          \end{scope}
      \end{scope}

      \draw[->,thick] (-6,0) -- (6,0) node[anchor=west] {\tiny{$x$}};

      \draw[->,thick] (-4,-0.4) -- (-4,5.5);

%
%

\begin{scope}[draw=gray,rotate=45,scale=sqrt(2)]

          \draw[->,thick] (0,0) -- (6,0);

          \draw[->,thick] (0,0) -- (0,5);

          \draw[ultra thick,draw=black] (5.5,0) -- (4,0)  -- (4,2) -- (2,2) -- (2,3) --
          (0,3) -- (0,4.5) ;
          \fill[fill=gray,opacity=0.1] (4,0) -- (4,2) -- (2,2) -- (2,3) -- (0,3) -- (0,0) -- cycle ;

      \end{scope}

 \draw[ultra thick,dotted] (4,4) -- (4,0) node[anchor=north] {\tiny{$x_1=8$}};
 \draw[ultra thick,dotted] (2,6) -- (2,-.8) node[anchor=north] {\tiny{$x_2=6$}};
 \draw[ultra thick,dotted] (0,4) -- (0,0) node[anchor=north] {\tiny{$x_3=4$}};
 \draw[ultra thick,dotted] (-1,5) -- (-1,-.8) node[anchor=north] {\tiny{$x_4=3$}};
 \draw[ultra thick,dotted] (-3,3) -- (-3,0) node[anchor=north] {\tiny{$x_5=1$}};

\end{scope}

    \end{tikzpicture}\vspace{-.5cm} \]
    \caption{Young diagram $\lambda=(4,4,2)$,
     the graph of the associated function $\omega_\lambda$
     and a non-centered version of it.
}
    \label{FigRussian}
\end{figure}

These numbers $x_1,x_2,\cdots,x_{2m+1}$
are called {\em (Kerov) interlacing coordinates}, see, {\it e.g.},
\cite[Section 6 with $\theta=1$]{OlshanskiPlancherelAverage}.
They are usually labeled with two different alphabets for minima
and maxima, but we shall rather use the same alphabet here and distinguish
between odd-indexed and even-indexed variables when necessary.

Note that not any decreasing sequence of integers can be obtained in this
way, as interlacing coordinates always satisfy the relation
$\sum_i (-1)^i x_i=0$.
A way to have independent coordinates is to consider {\em non-centered} Young diagrams.
By definition, the piecewise affine function $\omega$ of a non-centered Young diagram
is given by $x \mapsto \omega_\lambda(x-c) $ for some integer $c$
and {\em usual} Young diagram $\lambda$ -- see the right-most part on Figure~\ref{FigRussian}.

A Young diagram can be easily recovered from its Kerov coordinates
$x_1, \dots,x_{2m+1}$. Take some decreasing integral sequence $x_1, \dots,x_{2m+1}$.
First compute $c=\sum_i (-1)^i x_i$.
Then, to obtain its border, first draw the half-line $y=-x+c$ for $x\leq x_{2m+1}$,
then, without raising the pen, draw line segments of slope alternatively $+1$
and $-1$ between points of $x$-coordinates $x_{2m+1},x_{2m},\dots,x_1$ and
finally a half-line of slope $+1$ for $x \geq x_1$.
This last half-line has equation $y=x-c$ and 
the resulting broken line is the border of a non-centered Young diagram.

Apply now the same process to a non-increasing sequence 
$x_1, x_2,\dots,x_{2m+1}$ such that $x_i=x_{i+1}$.
Reaching the $x$-coordinate $x_i=x_{i+1}$, one has to change twice the sign of
the slope, that is, to do nothing. Hence, one obtains the same diagram as for
sequence 
$x_1,\cdots,x_{i-1},x_{i+2},\cdots,x_{2m+1}$.
Indeed, the same value $c$ is associated with both sequences.
Therefore, if one wants to interpret a stable polynomial in $\Q[X]$
as a function of Young diagrams, it is natural to require that it satisfies
Equation~\eqref{eqfoncQS}.\medskip

Equations~\eqref{EqQZero} and \eqref{EqPZero} arise in a similar way if we consider
multirectangular coordinates of Young diagrams.
These coordinates were introduced by R.~Stanley\footnote{In fact, R.~Stanley considered coordinates $\pp'$ and $\qq'$ related to
ours by $p'_i=p_i$ and $q'_i=q_1+\dots+q_i$. However, for our purpose,
we prefer the more symmetric version presented here.}
in \cite{St1}.
Consider two sequences $\pp$ and $\qq$ of non-negative integers
of the same length $m$.
We associate with these the Young diagram drawn on the left-hand side of Figure \ref{FigYDMultirec}.

\begin{figure}[t]
    \[\begin{tikzpicture}[scale=.8]
           
           \begin{scope}[scale=.6]
           \draw (0,0)--(5,0)--(5,1.5)--(3.2,1.5)--(3.2,3)--(1.3,3)--(1.3,4.8)--(0,4.8)--(0,0);
            \draw[<->] (0,-.2)--(1.3,-.2) node [below,midway] {\footnotesize $q_3$};
            \draw[<->] (1.3,-.2)--(3.2,-.2) node [below,midway] {\footnotesize $q_2$};
            \draw[<->] (3.2,-.2)--(5,-.2) node [below,midway] {\footnotesize $q_1$};
            \draw[<->] (-.2,0)--(-.2,1.5) node [left,midway] {\footnotesize $p_1$};
            \draw[<->] (-.2,1.5)--(-.2,3) node [left,midway] {\footnotesize $p_2$};
            \draw[<->] (-.2,3)--(-.2,4.8) node [left,midway] {\footnotesize $p_3$};
            \end{scope}
        
\begin{scope}[xshift=7cm, yshift=-0cm, scale=0.5, font=\footnotesize ]
       \begin{scope}
          \clip (-5.5,0) rectangle (5.5,5.5);
          \draw[thin, dotted] (-6,0) grid (6,6);
          \begin{scope}[rotate=45,draw=gray,scale=sqrt(2)]
              \clip (0,0) rectangle (4.5,5.5);
              \draw[thin, dotted] (0,0) grid (6,6);
          \end{scope}
      \end{scope}

      \draw[->,thick] (-6,0) -- (6,0);

      \draw[->,thick] (-4,-0.4) -- (-4,5.5);

\begin{scope}[draw=gray,rotate=45,scale=sqrt(2)]

          \draw[->,thick] (0,0) -- (6,0);

          \draw[->,thick] (0,0) -- (0,5);

          \draw[ultra thick,draw=black] (5.5,0) -- (4,0)  -- (4,2) -- (2,2) -- (2,3) --
          (0,3) -- (0,4.5) ;
          \fill[fill=gray,opacity=0.1] (4,0) -- (4,2) -- (2,2) -- (2,3) -- (0,3) -- (0,0) -- cycle ;

      \end{scope}

\draw[<->] (-4,3.7) to node[anchor=north] {$p_3$} (-3, 2.7) ;
\draw[<->] (-3,3.3) to node[anchor=south] {$q_2$} (-1, 5.3) ;
\draw[<->] (-1,4.7) to node[anchor=north] {$p_2$} (0, 3.7) ;
\draw[<->] (0,4.3) to node[anchor=south] {$q_1$} (2, 6.3) ;
\draw[<->] (2,5.7) to node[anchor=north] {$p_1$} (4, 3.7) ;

\end{scope}

\begin{scope}[xshift=14cm, yshift=-0cm, scale=0.5, font=\footnotesize ]
       \begin{scope}
          \clip (-5.5,0) rectangle (5.5,5.5);
          \draw[thin, dotted] (-6,0) grid (6,6);
          \begin{scope}[rotate=45,draw=gray,scale=sqrt(2)]
              \clip (0,0) rectangle (4.5,5.5);
              \draw[thin, dotted] (0,0) grid (6,6);
          \end{scope}
      \end{scope}

      \draw[->,thick] (-6,0) -- (6,0);

      \draw[->,thick] (-4,-0.4) -- (-4,5.5);

\begin{scope}[draw=gray,rotate=45,scale=sqrt(2)]

          \draw[->,thick] (0,0) -- (6,0);

          \draw[->,thick] (0,0) -- (0,5);

          \draw[ultra thick,draw=black] (5.5,0) -- (4,0)  -- (4,2) -- (2,2) -- (2,3) --
          (0,3) -- (0,4.5) ;
          \fill[fill=gray,opacity=0.1] (4,0) -- (4,2) -- (2,2) -- (2,3) -- (0,3) -- (0,0) -- cycle ;

      \end{scope}

\draw[<->] (-4,3.7) to node[anchor=north] {$p_4$} (-3, 2.7) ;
\draw[<->] (-3,3.3) to node[anchor=south] {$q_3$} (-1, 5.3) ;
\draw[<->] (-1,4.7) to node[anchor=north] {$p_3$} (0, 3.7) ;
\draw[<->] (0,4.3) to node[anchor=south] {$q_2$} (1, 5.3) ;
\draw[<->] (1,5.3) to node[anchor=south] {$q_1$} (2, 6.3) ;
\draw[<->] (2,5.7) to node[anchor=north] {$p_1$} (4, 3.7) ;

\end{scope}

    \end{tikzpicture}\vspace{-.5cm} \]
    \caption{Multirectangular coordinates of Young diagrams.
}
    \label{FigYDMultirec}
\end{figure}
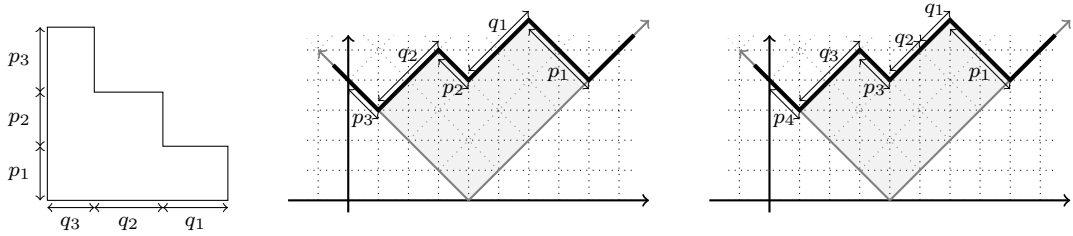

To construct non-centered Young diagrams, we introduce a new coordinate $p_{m+1}$,
which records the distance between the origin and the first corner of the diagram in Russian representation.
Unlike other multirectangular coordinates, $p_{m+1}$ can be negative
(in fact, $p_{m+1}$ simply corresponds to $x_{2m+1}$).
The central part of Figure \ref{FigYDMultirec} shows the multirectangular coordinates of 
a non-centered Young diagram.

Note that we allow some $p_i$ or some $q_i$ to be zero, so that the
same diagram can correspond to several sequences.
The right-hand side part of Figure~\ref{FigYDMultirec} shows
a different set of multirectangular (with $p_2=0$) associated
with the same non-centered Young diagram.

Equations~\eqref{EqQZero} and \eqref{EqPZero} exactly translate the fact
that a polynomial in multirectangular coordinates only depend on the underlying
Young diagrams and not on the chosen set of multirectangular coordinates.\medskip

To conclude, observe that multirectangular coordinates are related to interlacing coordinates by the
following linear changes of variables: for all $i \le m$,
\begin{equation}\label{EqChgtVar}
    \left\{ \begin{array}{l}
        p_i = x_{2i-1}-x_{2i};\\
        q_i = x_{2i}-x_{2i+1};\\
        p_{m+1} = x_{2m+1} ;
    \end{array} \right. \hspace{1cm}
     \left\{ \begin{array}{l}    
        x_{2i+1} = q_{i+1}+\dots+q_m + p_{i+1}+ \dots + p_{m+1}; \\
        x_{2i} = q_i+\dots+q_m+p_{i+1}+\dots+p_{m+1}.
    \end{array} \right.
\end{equation}

\subsection{Solution of the first equation}

\begin{theorem}\label{ThmSolEqQSym}
A stable polynomial $f=(f_{2m+1})_{m \ge 0}$ satisfies the functional equation \eqref{eqfoncQS}
if and only if there exists $F \in\QSym$ such that 
$ f_{2m+1}(x_1, \cdots, x_{2m+1}) = F(\X_m)$.
\end{theorem}

\begin{proof}
Note that a polynomial $f$ satisfies Equation \eqref{eqfoncQS}
if and only if all its homogeneous components do.
Therefore it is enough to prove the statement for a homogeneous function $f$.
\smallskip

Let us first prove that the dimension of the space of homogeneous polynomials
in $\Q[X]$ of degree $n$ satisfying \eqref{eqfoncQS} is at most equal to
$2^{n-1}$.
We say that a monomial 
$X^\bv=x_1^{v_1}x_2^{v_2}\cdots$
in $\Q[X]$ is {\em packed} if $\bv$ can be written as $\bc,0,0,0,\dots$
with $\bc$ a composition (\ie a vector whose entries are positive integers).
Thus, the number of packed monomials of degree $n$ is $2^{n-1}$.
Let 
$P=\sum_{\bv} c_\bv X^\bv$ 
be a homogeneous polynomial of degree $n$, which is solution of
\eqref{eqfoncQS} (here, the sums runs over sequences of
non-negative integers of sum $n$).
Associate the integer $\ell(X^\bv)=\sum_{i \ge 1}iv_i$ 
to a monomial~$X^\bv$.

Then, we claim that all the coefficients of $P$ are determined by those of packed monomials.
To see this, consider a non-packed monomial $X^\bw=x_1^{w_1}x_2^{w_2}\cdots$ 
with $w_i=0$ and $w_{i+1}\neq 0$. 
We substitute $x_i=x_{i+1}=x$ in $P$.
Looking at the monomial
$$x_1^{w_1}\cdots x_{i-1}^{w_{i-1}}x^{w_{i+1}}x_{i+2}^{w_{i+2}}\cdots$$
that does not appear on the right-hand side of \eqref{eqfoncQS}, 
we get a linear relation between $c_\bw$ and the coefficients $c_\bv$ of the
monomials $X^\bv$ such that $\ell(X^\bv)<\ell(X^\bw)$,
whence the upper bound on the dimension.
\medskip

Now, we clearly have:
\[\left.\prod_{1 \le j \le m+1}^\rightarrow \sigma_{x_j}(A)^{(-1)^j}
\right|_{x_i=x_{i+1}} = 
\prod_{\gf{1 \le j \le m+1}{j \neq i,\ j \neq i+1}}^\rightarrow \sigma_{x_j}(A)^{(-1)^j}. \]
Looking at Equation \eqref{eq:def_MIXm} and using the fact that
 $S^I(A)$ is a linear basis of $\Sym(A)$, we get that, for each composition $I$,
the stable polynomial $(M_I(\X_m))_{m \ge 1}$ satisfies Equation \eqref{eqfoncQS}.
 Moreover, all $M_I(\X)$ are linearly
independent since, setting $x_{2i+1}=0$ in $\X$ transforms $M_I(\X)$ into
the usual monomial quasi-symmetric functions in even-indexed variables
$M_I(x_2,x_4,x_6,\cdots)$.
We have found $2^{n-1}$ linearly independent solutions of Equation \eqref{eqfoncQS}
in degree $n$, which finishes the proof.
\end{proof}\vspace{-.2cm}

\subsection{Bijection between the spaces of solutions}
\begin{proposition}
The substitution~\eqref{EqChgtVar} defines an algebra isomorphism
between $\Sx$ and $\Spq$.
\end{proposition}
\begin{proof}
Consider an element $f=(f_{2m+1})_{m \ge 0}$ in $\Sx$.
Let $m\ge 0$.
Replace all variables $x_1,\ldots,x_{2m+1}$ in $f_{2m+1}$
according to \eqref{EqChgtVar}, and set
\begin{equation}
    \label{eq:def_h_from_f}
h_m\left( \begin{array}{cccc} 
        p_1 & \dots & p_m & p_{m+1} \\
        q_1 & \dots & q_m
    \end{array} \right) = 
    f_{2m+1}(x_1,\dots,x_{2m+1}).
\end{equation}

Clearly, $h_m$ is a polynomial in $p_1,\dots,p_m,p_{m+1},q_1,\dots,q_m$.
Moreover, by definition,
\[h_{m+1}\left( \begin{array}{ccccc}                                  
        p_1 & \dots & p_m &p_{m+1} &0 \\                                    
        q_1 & \dots & q_m &0                                      
    \end{array} \right) =
f_{2m+3}(x_1,\dots,x_{2m+1},0,0).\]\pagebreak

But, as $f$ is a stable polynomial, the right-hand side is equal to
$f_{2m+1}(x_1,\dots,x_{2m+1})$.
Thus Equation \eqref{eq:def_h_from_f} implies that $(h_m)_{m \ge 0}$ is an element of $\Q[\bm{p},\bm{q}]$.

We will now show that $h$ satisfies Equation \eqref{EqQZero}.
Let us now consider an integer $m \ge 1$ and variables
$p_1,\dots,p_m,p_{m+1},q_1,\dots,q_m$.
Assume additionally that $q_i=0$ for some $i$, which implies $x_{2i}=x_{2i+1}$
Thus, as $f$ is an element of $\Sx$, we have 
\[f_{2m+1}(x_1,\dots,x_{2m+1})=f_{2m-1}(x_1,\dots,x_{2i-1},x_{2i+2},\dots,x_{2m+1}).\]
Observe that the right-hand side corresponds to the definition of
 \[   h_{m-1}\left( \begin{array}{cccccccc}
     p_1&\dots&p_{i-1}&p_i+p_{i+1}&\dots&p_m & p_{m+1}\\
        q_1&\dots&q_{i-1}&q_{i+1}&\dots&q_m
    \end{array} \right), \]
which ends the proof of Equation \eqref{EqQZero}.
Equation \eqref{EqPZero} can be proved in a similar way.

Finally, from a stable polynomial $f$ in $\Sx$, we have constructed
an element $h=(h_m)_{m \ge 0}$ in $\Spq$.
This map from $\Sx$ to $\Spq$
is clearly an algebra morphism.\medskip

Its inverse can be constructed also by using Equation \eqref{EqChgtVar},
which proves that it is an isomorphism.
\end{proof}\vspace{-.7cm}

\begin{remark}
In light of the interpretation of $\Sx$ and $\Spq$ in terms of Young diagrams,
this isomorphism is not surprising, as both equations are in fact the same,
written with different sets of coordinates.
\end{remark}\vspace{.2cm}

\section{Generating functions of $P$-partition}
\label{sec:P-part}

Let us come back to our problem of $P$-partition.
If $P$ is a ranked poset, definitions of $F_P$ and $N_P$
were given in the introduction.
Let us illustrate these with an example.

\begin{example}{\rm
    Consider the poset $P_\ex$ drawn on the left-hand side of
    Figure \ref{fig:example_poset}.
    Let $r$ be a function from $P_{\text{ex}}$ to $\N$.
    Denote by $e$ and $f$ the images of the leftmost white elements, 
    by $g$ and $h$ the images of the black elements just to their right,
    then by $i$ the image of the white element to their right
    and finally by $j$  the image of the rightmost black element.
    Then, by definition, $r$ satisfies the order condition if and only if
    $e,f \le g,h < i \le j$.
Note the alternating large and strict inequalities.
Finally, one has
\begin{align*}
F_{P_\ex}(x_1,x_2,\dots) &= \! \sum_{e,f \le g,h < i \le j} \!
        x_e \, x_f \, x_g \, x_h \, x_i \, x_j ;\\
N_{P_\ex}\left( \begin{array}{ccc}
        p_1 & p_2 & \dots\\       
            q_1 & q_2 & \dots         
        \end{array} \right) &= \! \sum_{e,f \le g,h < i \le j} \!
        p_e \, p_f \, p_i \, q_g \, q_h \, q_j.
    \end{align*}
    }
\end{example}

    \begin{figure}
        \[\begin{tikzpicture}
            \tikzstyle{bv}=[circle,fill=black,inner sep=0pt,minimum size=2mm]
            \tikzstyle{wv}=[circle,draw=black,inner sep=0pt,minimum size=2mm]
            \node[wv] (v1) at (0,0) { };
            \node[wv] (v2) at (0,-1) { };
            \node[bv] (v3) at (1,0) { };
            \node[bv] (v4) at (1,-1) { };
            \node[wv] (v5) at (2,-.5) { };
            \node[bv] (v6) at (3,-.5) { };
            \draw[->] (v1) -- (v3);
            \draw[->] (v1) -- (v4);
            \draw[->] (v2) -- (v3);
            \draw[->] (v2) -- (v4);
            \draw[->] (v3) -- (v5);
            \draw[->] (v4) -- (v5);
            \draw[->] (v5) -- (v6);
        \end{tikzpicture}\qquad \qquad
        \begin{tikzpicture}[font=\scriptsize]
            \tikzstyle{bv}=[circle,fill=black,inner sep=0.1mm,minimum size=2mm]
            \tikzstyle{wv}=[circle,draw=black,inner sep=0.1mm,minimum size=2mm]
            \node[wv] (v1) at (0,0) {$2$};
            \node[wv] (v2) at (0,-1) {$3$};
            \node[bv] (v3) at (1,0) {\color{white} $1$};
            \node[bv] (v4) at (1,-1) {\color{white} $5$};
            \node[wv] (v5) at (2,-.5) {$6$};
            \node[bv] (v6) at (3,-.5) {\color{white} $4$};
            \draw[->] (v1) -- (v3);
            \draw[->] (v1) -- (v4);
            \draw[->] (v2) -- (v3);
            \draw[->] (v2) -- (v4);
            \draw[->] (v3) -- (v5);
            \draw[->] (v4) -- (v5);
            \draw[->] (v5) -- (v6);
        \end{tikzpicture}\]
        \caption{Example of an unlabeled ranked poset $P$, 
        and a labeled version of it.
        We represent here the graph of its covering relation
        that is its {\em Hasse diagram}.
        Edges are oriented from left to right.
Vertices in $V_0$ (resp. $V_1$) 
    are drawn in white (resp. black).}
        \label{fig:example_poset}
    \end{figure}
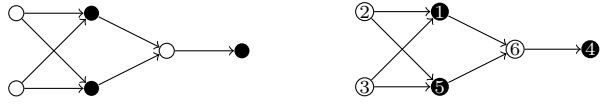
\begin{lemma}
\label{LemNGInSol}
Let $P$ be a ranked poset.
Then the stable polynomial $N_P$ belongs to $\Spq$.
\end{lemma}

\begin{proof}
    Let us check that $N_P$ satisfies Equation~\eqref{EqQZero}. 
    We define 
    \[ \begin{cases}
        p'_j =p_j & \text{ if }j<i ;\\
        p'_i =p_i+p_{i+1} ;& \\
        p'_j =p_{j+1} & \text{ if }j>i ;
    \end{cases} \qquad
    \begin{cases}  
        q'_j =q_j & \text{ if }j<i ;\\
        q'_j =q_{j+1} & \text{ if }j \ge i.
    \end{cases}\]
    These are the variables in the right-hand side of~\eqref{EqQZero}.
    Consider first the left-hand side:
    \[N_P \left.\left( \begin{array}{cccc}                               
            p_1 & \dots & p_m & p_{m+1}\\                                       
                q_1 & \dots &q_m                                       
            \end{array} \right)\right|_{q_i=0} = \sum_{r} \left( \prod_{v_0 \in V_0} p_{r(v_0)} 
            \prod_{v_1 \in V_1} q_{r(v_1)} \right),\]
where the sum runs over functions $r : P \to \{1,\dots,m+1\}$ satisfying the order condition
(and such that $r(v_1) \neq m+1$ for $v_1 \in V_1$).
Additionnally, one can restrict the sum to functions $r$ such that
$r(v_1) \neq i$ for any $v_1 \in V_1$
(we call these $i$-avoiding functions).

With any function $r$, we associate a function $r'=\Phi(r) : V \to \{1,\dots,m\}$ defined as follows:
\[r'(v) = r(v) \text{ if }r(v) \le i \text{ and } r'(v) = r(v)-1 
\text{ if }r(v) > i.\]
It is straightforward to check that,
if $r$ is $i$-avoiding and satisfies the order condition,
then $r'$ also satisfies the order condition.
Indeed, the only problem which could occur is $r(v_1)=i$ and $r(v_0)=i+1$
for a covering relation $v_1 <_P v_0$ with $v_1 \in V_1$,
which cannot happen since we forbid $r(v_1)=i$.

The preimage  of a given function $r'$ is obvious:
it is the set of functions $r$ such that
\[\begin{cases}
    r(v) = r'(v) & \text{ if }r'(v) < i ;\\
    r(v) \in \{i;i+1\} & \text{ if }r'(v) = i ;\\
    r(v)= r'(v)+1 & \text{ if }r'(v) > i.
\end{cases}\]
If $r'$ satisfies the order condition, all its $i$-avoiding pre-images $r$ 
also satisfy the order condition.
Once again, the only obstruction to this would be the case $r'(v_0)=r'(v_1)=i$
for some covering relation $v_0 <_P v_1$ with $v_0 \in V_0$.
In this case, preimages $r$ with $r(v_0)=i+1$ and $r(v_1)=i$ would not
satisfy the order condition but we forbid $r(v_1)=i$.
Now,
for any function $r'$, one has:
\[\sum_{r \in \Phi^{-1}(r')} \left( \prod_{v_0 \in V_0} p_{r(v_0)}      
\prod_{v_1 \in V_1} q_{r(v_1)} \right)
= \left( \prod_{v_0 \in V_0} p'_{r'(v_0)}      
\prod_{v_1 \in V_1} q'_{r'(v_1)} \right),\]
Summing over all functions $r': V \to \{1,\dots,m-1\}$ with the order condition,
we get equality~\eqref{EqQZero}.

The proof of \eqref{EqPZero} is similar.
Only the case $i=1$ is slightly different, but it is straightforward to check that
every monomial containing $q_1$ also contains $p_1$, which proves \eqref{EqPZero}
in the case $i=1$.
\end{proof}

\begin{remark}
{\rm
In~\cite[Section 1.5]{FS11}, an equivalent definition of $N_P$ as
a function on Young diagrams is given in the case of posets of height $1$
 (which correspond to bipartite graphs).
The fact that $N_P$ can be defined using only the Young diagram and
not its multirectangular coordinates explains that it belongs to $\Spq$.
}
\end{remark}

We can now prove Theorem \ref{thm:main}.

{\bf Proof of Theorem \ref{thm:main}:}
As $N_P$ belongs to $\Spq$, if we express it in terms of the variables $x_1,x_2,\dots$,
we get an element of $\Sx$.
By Theorem \ref{ThmSolEqQSym}, it is equal to $(F(\X_m))_{m \ge 0}$
for some quasi-symmetric function.

To identify $F$, we shall send all odd-indexed variables $x_{2i+1}$ to $0$.
This amounts to sending $p_i$ to $-x_{2i}$ (for $1 \le i \le m$),
$p_{m+1}$ to $0$ and $q_i$ to $x_{2i}$.
Therefore, under this substitution,
\[N_P \left( \begin{array}{ccccc}
    p_1 & \dots & p_m & p_{m+1} & 0\\
    q_1 & \dots & q_m & 0
\end{array} \right) = (-1)^{|V_0|} F_P(x_2,x_4,\cdots,x_{2m}).\]
On the other hand, it is straightforward to check from \eqref{eq:def_MIXm}
that, under the substitution $x_{2i+1}=0$
\[F(\X_m) = F(x_2,x_4,\cdots,x_{2m}).\]
As both equations are true for all values of $m$,
we get $F=(-1)^{|V_0|} F_P$ and the theorem follows.
\qed

\section{Noncommutative generalization}
\label{sec:NC}
\subsection{Noncommutative analog of the main theorem}

The natural noncommutative analogue of $\QSym$ is the algebra of
\emph{word quasi symmetric functions}, denoted by $\WQSym$.
We refer to the long version \cite{LongVersion} for basic facts about 
$\WQSym$ and for the definition of the evaluation of a function in $\WQSym$
on the virtual alphabet $\A_m$ defined below.

Take as data a ranked poset $\bm{P}$,
whose set of elements 
$V= V_0 \sqcup V_1$ is equal to $\{ 1, \dots, n \}$.
Then we define the non-commutative analog $\bm{F_P}$ of $F_P$ as follows:
\[\bm{F_P}(a_1,a_2,\dots)
  =\sum_{\gf{r:V \to \N}{\text{with order condition}}}
a_{r(1)} a_{r(2)} \dots a_{r(n)}.\]
Here, the $a_i$ are noncommuting variables.
Then $\bm{F_P}$ is a word quasi-symmetric function.

In the same way, we can define a noncommutative analog $\bm{N_P}$ of $N_P$:
\[\bm{N_P}\left( \begin{array}{ccc}
            b_1 & b_2 & \dots\\
                d_1 & d_2 & \dots
            \end{array} \right)=
\sum_{\gf{r:V \to \N}{\text{with order condition}}}
\bd_{r(1)} \bd_{r(2)} \dots \bd_{r(n)},\]
where we use the 
shorthand notation $\bd_{r(i)}=b_{r(i)}$ for $i \in V_0$
and $\bd_{r(i)}=d_{r(i)}$ for $i \in V_1$.

\begin{example}\label{ex:NG}{\rm
    Consider the ranked poset $\bm{P_\ex}$ drawn on the right-hand side
    of Figure \ref{fig:example_poset}.
    This is a labeled version of the poset $P_\ex$ on
    the left-hand side of the same Figure.

    Let $r$ be a function from its element set, that is $\{1,\cdots,6\}$ to $\N$.
    Define 
    \[e:=r(2),\ f:=r(3),\ g:=r(1),\ h:=r(5),\ i:=r(6),\ j:=r(4).\]
    Then, by definition, $r$ satisfies the order condition if and only if
    $e,f \le g,h < i \le j,$
so one has
\[\bm{N_{P_\ex}}\left( \begin{array}{ccc}
        b_1 & b_2 & \dots\\       
            d_1 & d_2 & \dots         
        \end{array} \right) = \sum_{e,f \le g,h < i \le j}
        d_g \, b_e\, b_f\, d_j\, d_h\, b_i,\]
 which is a noncommutative version of the function $N_{P_\ex}$
 given in Example \ref{ex:NG}.
    }
\end{example}

We shall now present a noncommutative analog of our main theorem.
Using two alphabets of noncommuting variables $b_i$ and $d_i$,
define the virtual alphabet
$
\A_m = \ominus (a_1) \oplus (a_2) \ominus (a_3) \cdots \ominus (a_{2m+1}), 
$
where the $a_i$ are the following linear combinations of $b_i$ and $d_i$:
\begin{equation}\label{eq:a_bd}
\left\{ \begin{array}{l}
a_{2i+1} = d_{i+1} + \cdots + d_m + b_{i+1} + \cdots + b_{m+1} ;  \\
a_{2i} = d_{i} + \cdots + d_m + b_{i+1} + \cdots + b_{m+1}.
\end{array} \right.
\end{equation}

\begin{theorem}
\label{thm:main_NC}
For any labeled ranked poset $\bm{P}$ with element set 
    $V=V_0 \sqcup V_1= \{1,\dots,n\}$ and any integer $m$
    \[ \bm{N_P}\left( \begin{array}{cccc}
            b_1 &  \dots & b_m &b_{m+1} \\
                d_1 & \dots & d_m
            \end{array} \right) = (-1)^{|V_0|} \big(\bm{F_P}(\A_m)\big).\]
\end{theorem}
\begin{proof}
Omitted for brevity. The main lines of the proof are the same than in the commutative framework:
defining two equivalent functional equations, solving the first one and showing that
$\bm{N_P}$ is a solution of the second one. 

A difficulty is that there is no general theory
of difference of alphabets in $\WQSym$ (its antipode is not involutive), so we had to find an ad-hoc
definition for $\bm{F_P}(\A_m)$ so that the theorem holds,
see \cite[Section 5]{LongVersion} for details.
 \end{proof}

\subsection{Noncommutative Luoto basis}
Consider the following family of ranked poset:
let $\bm{K}$ be a set-composition (that is an ordered set partition) $(K_0, \dots,K_{\ell-1})$ of $\{1,\dots,n\}$,
then we define the labeled poset $\bm{P_K}$ by its set of covering relations:
\[(K_0 \times K_1) \sqcup  (K_1 \times K_2) \sqcup\dots \sqcup
 (K_{\ell-2} \times K_{\ell-1}).\]
 Then $\bm{P_K}$ is ranked and elements of $K_i$ have height $i$.
 Thus
        \[
        V_0=K_0 \sqcup K_2 \sqcup \cdots  \text{ and }
        V_1=K_1 \sqcup K_3 \sqcup \cdots .
        \]

\begin{example}
    If $\bm{K}=\big( \{2,3\},\{1,5\},\{6\},\{4\} \big)$,
    then $\bm{P_K}$ is the graph of the right-hand side
    of Figure \ref{fig:example_poset}.
\end{example}

\begin{proposition}
    The functions $\bm{F_{P_K}}$, where $\bm{K}$ runs over set compositions,
    form a $\Z$-basis of $\WQSym$.
   \end{proposition}
\begin{proof}
    By a dimension argument, proving that $\bm{F_{P_K}}$ is a basis
    of $\WQSym$ reduces to proving the linearly independence.
But, thanks to Theorem \ref{thm:main_NC}, one can prove instead that
$(\bm{N_{P_K}})$ is linear independent.

    With a noncommutative monomial (a word) in $b_i$ and $d_i$,
    we can associate its {\em evaluation}, which we define as
    the integer sequence
    (number of $b_1$, number of $d_1$, number of $b_2$, $\dots$).
    It is immediate to see that the monomial in $\bm{N_{P_K}}$
    with the lexicographically largest evaluation is obtained as follows:
    it has letters $b_1$ in positions given by $K_0$, letters $d_1$
    in position given by $K_1$, letters $b_2$ in positions given by $K_2$, and so on.
    It follows that the set-composition $\bm{K}$ can be recovered from
    the monomial of lexicographically largest evaluation in $\bm{N_{P_K}}$,
    which implies the linear independence of the $\bm{N_{P_K}}$.

    Now, if $\bm{F} \in \WQSym$ has integer coefficients in $a$,
    then $\bm{F}(\A_m)$ has integer coefficients in $b$ and $d$
    and the argument above shows that it is an integral linear combination of
    the $\bm{N_{P_K}}$.
    But, as in the proof of Theorem~\ref{thm:main},
    substituting $a_{2i+1}=0$ sends back $\bm{F}(\A_m)$ to $\bm{F}$
    and $N_{P_K}$
    to $(-1)^{|V_0|} \bm{F_{P_K}}$, showing that the latter is a $\Z$-basis of $\QSym$.
\end{proof}

This basis $\bm{F_{P_K}}$ is a natural noncommutative analog of a basis studied by K.~Luoto
\cite{Luoto}.
We have not been able to prove the linear independence of $\bm{F_{P_K}}$ without
using our theorem to duplicate the alphabet.
In particular, Luoto's proof to show that his basis is indeed linearly independent
does not seem to extend to the noncommutative framework.
The following properties illustrate, in our opinion, the relevance of this new basis of $\WQSym$.

\begin{proposition}\label{prop:positive_expansion}
For any labeled rank poset $\bm{P}$,
the function $\bm{F_{P}}$ expands as a linear combination with 
{\em nonnegative integer coefficients} of $\bm{F_{P_K}}$.
\end{proposition}
{\bf Sketch of proof.}
Let $r$ be a function from $\bm{P}$ to $\N$ satisfying the order condition.
For any pair $(x,y)$ in $V_0 \times V_1$ such that $x$ and $y$ are incomparable in $P$
one has either $r(x) \leq r(y)$ or $r(y) < r(x)$.
Let us split the sum in \eqref{eq:def_FP} depending on
 the set of such pairs $(x,y)$ such that $r(x) \leq r(y)$.
 Then the nonempty sums correspond to some $F_{\bm{Q}}$,
 where $\bm{Q}$ is the ranked poset in which any element of odd height is comparable
 with any element of even height.
 Such posets are exactly the $\bm{P_K}$, which conclude the proof.
\qed

\begin{corollary}
The multiplication table of the basis $\bm{F_{P_K}}$ has nonnegative integer entries.
\end{corollary}

\begin{proof}
The product $\bm{F_{P_K}} \cdot \bm{F_{P_{K'}}}$ is simply $\bm{F_{P_K \sqcup P_{K'}}}$.
Proposition~\ref{prop:positive_expansion} ends the proof.
\end{proof}

\end{document}